\documentclass[11pt,reqno]{amsart}

\usepackage{amsmath,amsthm,amssymb,enumerate}
\usepackage{soul}
\usepackage{color}
\usepackage[breaklinks=true]{hyperref}

\theoremstyle{plain}
\newtheorem{theorem}{Theorem}[section]

\newtheorem{proposition}[theorem]{Proposition}

\theoremstyle{definition}
\newtheorem{definition}[theorem]{Definition}
\newtheorem{remark}[theorem]{Remark}

\newtheorem{example}[theorem]{Example}

\numberwithin{equation}{section}

\makeatletter\let\c@theorem\c@table\makeatother

\newcommand{\alg}{\mathcal{A}}
\newcommand{\Aut}{\mathop{\mathrm{Aut}}}

\newcommand{\be}{\mathbf{e}}
\newcommand{\bJ}{\mathbf{1}}

\newcommand{\br}{\mathbf{r}}
\newcommand{\bs}{\mathbf{s}}

\newcommand{\bu}{\mathbf{u}}
\newcommand{\bv}{\mathbf{v}}
\newcommand{\bw}{\mathbf{w}}
\newcommand{\bx}{\mathbf{x}}
\newcommand{\by}{\mathbf{y}}
\newcommand{\bZ}{\mathbf{0}}
\newcommand{\C}{\mathbb{C}}
\newcommand{\cC}{\mathcal{C}}
\newcommand{\cF}{\mathcal{F}}
\newcommand{\cP}{\mathcal{P}}
\newcommand{\cW}{\mathcal{W}}
\newcommand{\comp}{\mathbin{\circ}}
\newcommand{\diag}{\mathop{\mathrm{diag}}\nolimits}

\newcommand{\fG}{\mathfrak{G}}

\newcommand{\im}{\mathop{\mathrm{im}}}
\newcommand{\Id}{\mathrm{Id}}
\newcommand{\one}[1]{\bJ_{#1 \times #1}}
\newcommand{\R}{\mathbb{R}}
\newcommand{\rd}{\mathrm{d}}

\newcommand{\spec}{\sigma}
\newcommand{\std}{\,\rd}
\newcommand{\stratumsymb}{\mathcal{S}}
\newcommand{\stratum}[1]{\stratumsymb_{#1}}
\newcommand{\wg}{\cW_\fG}

\newcommand{\up}{\Sigma^\uparrow_\pi}
\newcommand{\down}{\Sigma^\downarrow_\pi}
\newcommand{\ccup}{\Theta^\uparrow_\pi}
\newcommand{\ccdown}{\Theta^\downarrow_\pi}

\begin{document}
\title{Matrix compression along isogenic blocks}

\author[A.~Belton]{Alexander Belton$^*$}
\address{$^*$Lancaster University, Lancaster, UK.
{\tt a.belton@lancaster.ac.uk}}

\author[D.~Guillot]{Dominique Guillot$^\dag$}
\address{$^\dag$University of Delaware, Newark, DE, USA.
{\tt dguillot@udel.edu}}

\author[A.~Khare]{Apoorva Khare$^\ddag$}
\address{$^\ddag$Indian Institute of Science, Bangalore, India.
{\tt khare@iisc.ac.in}}

\author[M.~Putinar]{Mihai Putinar$^\mathsection$}
\address{$^\mathsection$University of California at Santa Barbara, CA,
USA and Newcastle University, Newcastle upon Tyne, UK.
{\tt mputinar@math.ucsb.edu}, {\tt mihai.putinar@ncl.ac.uk}}

\date{17th March 2022}

\begin{abstract}
A matrix-compression algorithm is derived from a novel isogenic block
decomposition for square matrices. The resulting compression and
inflation operations possess strong functorial and spectral-permanence
properties. The basic observation that Hadamard entrywise functional
calculus preserves isogenic blocks has already proved to be of
paramount importance for thresholding large correlation matrices. 
{ The proposed isogenic stratification of the set of 
complex matrices bears similarities to the Schubert cell stratification of a homogeneous algebraic 
manifold}.
An
array of { potential} applications to current investigations in computational
matrix analysis is {briefly mentioned}, touching concepts such as symmetric
statistical models, hierarchical matrices and coherent matrix
organization induced by partition trees.
\end{abstract}



\maketitle

\tableofcontents

\section{Introduction}

\subsection{Prelude}\label{Sprelude}

In a previous paper \cite{BGKP-fixeddim} we were concerned with the
study of entrywise positivity preservers acting on the cone of $N
\times N$ positive semidefinite matrices. By a classical observation
of Loewner, as developed by Horn in \cite[Theorem~1.2]{horn}, if the
real-valued function $f$ defined on the positive half line preserves
positive semidefiniteness for all $N \times N$ matrices via entrywise
Hadamard calculus (for brevity, we say that $f$
\emph{preserves positivity}), then $f$ is necessarily of class
$C^{(N - 3)}$ and has non-negative derivatives up to this order.
Moreover, the derivatives of order up to $N - 1$ are non-negative if
they exist. In the case where $f$ is known to be analytic, preserving
positivity only for rank-one $N \times N$ matrices already implies
that the first $N$ non-zero Taylor coefficients of $f$ are strictly
positive. In particular, this holds for $f$ a polynomial. { Henceforth $A^{\circ k}$
denotes the entrywise (fractional) power of a matrix $A$}.

The central result in our previous work \cite{BGKP-fixeddim}
(subsequently refined by one of us with Tao \cite{KT}) is a closed
form for the greatest lower bound for the only possible negative
coefficient $c_M$ in the linear pencil
\begin{equation}\label{Epencil}
p( A ) = c_0 A^{\circ n_0} + c_1 A^{\circ n_1} + \cdots + %
c_{N - 1} A^{\circ n_{N - 1}} + c_M A^{\circ M},
\end{equation}
with real powers $n_0 < \cdots < n_{N - 1} < M$ lying in
$\{ 0, 1, 2, \ldots \} \cup [ N - 2, \infty )$, such that $p( A )$ is
positive semidefinite for any positive semidefinite $N\times N$ matrix
$A$ with entries in a specific domain in $\C$. The threshold value for
the coefficient~$c_M$ was obtained via a combinatorial formula
involving Schur polynomials.

A more natural matrix-theoretic approach to obtain the greatest lower
bound for $c_M$ in (\ref{Epencil}) is to study the spectrum of the
Rayleigh quotient~$R$ of two quadratic forms,
\begin{equation}\label{ErayleighIntro}
R = R( A, \bu ) := \frac{\bu^* A^{\circ M} \bu}%
{\bu^* ( c_0 A^{\circ n_0} + c_1 A^{\circ n_1} + \cdots + %
c_{N - 1} A^{\circ n_{N - 1}} ) \bu}.
\end{equation}
Finding the maximum of $R$ with respect to both $\bu$ and $A$ indeed
yields the critical threshold for preserving positivity as mentioned
above; we refer the reader to~\cite[Section~4]{BGKP-fixeddim}
and~\cite[Proposition 11.2]{KT}. One first obtains a closed-form
expression for $R_A := \sup_{\bu \in \R^N} R( A, \bu )$ and then
maximizes with respect to $A$. However, a significant difficulty
arises when using this approach: 

{\it the function $A \mapsto R_A$ is not
continuous on the set of $N \times N$ positive semidefinite
matrices.}

 As shown in \cite{BGKP-fixeddim}, discontinuities occur
precisely where constant-block structures emerge in $A$. The discovery
of this phenomenon led the authors to investigate, in
\cite[Section~5]{BGKP-fixeddim}, the simultaneous kernel
$\cap_{k \geq 0} A^{\circ k}$ of the Hadamard powers of an arbitrary
positive semidefinite matrix $A$, and a related Schubert cell-type
stratification of the set of all positive semidefinite matrices of a
given size.

A series of remarkable properties of this stratification is unveiled
in the present work. In order to facilitate access to the concepts and
notation in subsequent sections, we start by explaining informally the
main ideas with the help of a simple example. Consider the
$6 \times 6$ real symmetric matrix
\[
A := \begin{pmatrix}
\phantom{-}1 & \phantom{-}1 & \phantom{-}2 & -3 & \phantom{-}2 & \phantom{-}1\\
\phantom{-}1 & \phantom{-}1 & \phantom{-}2 & -3 & \phantom{-}2 & \phantom{-}1\\
\phantom{-}2 & \phantom{-}2 & \phantom{-}6 & -5 & \phantom{-}6 & \phantom{-}2\\
-3 & -3 & -5 & \phantom{-}10 & -5 & -3\\
\phantom{-}2 & \phantom{-}2 & \phantom{-}6 & -5 & \phantom{-}6 & \phantom{-}2\\
\phantom{-}1 & \phantom{-}1 & \phantom{-}2 & -3 & \phantom{-}2 & \phantom{-}1
\end{pmatrix}.
\]
Each row and column of $A$ is constant on the sets in the partition
\[
\pi = \bigl\{ \{ 1, 2, 6 \}, \ \{ 3, 5 \}, \ \{ 4 \} \bigr\}
\]
of the index set $\{ 1, 2, 3, 4, 5, 6 \}$. This observation gives at
once some null vectors for the matrix $A$: those annihilated by the
pattern of the partition, $( 1, 1, 0, 0, 0, 1 )$,
$( 0, 0, 1, 0, 1, 0 )$ and $( 0, 0, 0, 1, 0, 0 )$. In fact, the kernel
of every positive Hadamard power of $A$, including~$A$ itself, is
precisely this set of vectors.

Given any matrix $A' \in \C^{N \times N}$, there is a unique coarsest
partition $\pi'$ of $\{ 1, \ldots, N \}$ such that $A'$ is constant on
the blocks determined by $\pi'$. We denote by $\stratum{\pi'}$ the set
of all matrices having this constant-block structure. This naturally
yields a stratification of the space of complex matrices,
\begin{equation}\label{stratification}
\C^{N \times N} = \bigsqcup_{\pi'} \stratum{\pi'},
\end{equation}
where the disjoint union is taken over all partitions of
$\{ 1, \ldots, N \}$. We refer to each set $\stratum{\pi'}$ as a
\emph{stratum}, and call the above decomposition the
\emph{isogenic block stratification} of $\C^{N \times N}$.

We explore several applications of the stratification below. A natural
operation is the compression $\down$ of each constant block to a
single entry: with $A$ and $\pi$ as above,
\[
\down( A ) = \begin{pmatrix}
\phantom{-}1 & \phantom{-}2 & -3\\
\phantom{-}2 & \phantom{-}6 & -5\\
-3 & -5 & \phantom{-}10
\end{pmatrix}.
\]
This straightforward transformation has some remarkable and very
useful properties. For example, the spectrum
of~$D_\pi^{1 / 2} \down( A ) D_\pi^{1 / 2}$ has the same non-zero
eigenvalues, counting multiplicities, as the original matrix~$A$,
where the matrix~$D_\pi = \diag( 3, 2, 1 )$ simply reflects the size
of the blocks of~$\pi$. The spectrum of~$A$ has~$3$ additional zero
eigenvalues, corresponding to its kernel.

An even more striking use of compression is for the computation of the
Moore--Penrose generalized inverse: with $A$ as above, this is
\[
A^\dagger = \frac{1}{36} \begin{pmatrix}
\phantom{-}140 & \phantom{-}140 & -30 & \phantom{-}96 & -30 & \phantom{-}140\\
\phantom{-}140 & \phantom{-}140 & -30 & \phantom{-}96 & -30 & \phantom{-}140\\
-30 & -30 & \phantom{-}9 & -18 & \phantom{-}9 & -30\\
\phantom{-}96 & \phantom{-}96 & -18 & \phantom{-}72 & -18 & \phantom{-}96\\
-30 & -30 & \phantom{-}9 & -18 & \phantom{-}9 & -30\\
\phantom{-}140 & \phantom{-}140 & -30 & \phantom{-}96 & -30 & \phantom{-}140
\end{pmatrix}.
\]
In general, there is a fast method to compute the pseudo-inverse of a
block matrix:
\begin{equation}
A^\dagger = \up( D_\pi^{-1} \down( A )^{-1} D_\pi^{-1} ),
\end{equation}
where the inflation map $\up$ expands a matrix into one with blocks
which are constant for the index-set partition $\pi$.

While the isogenic block structure considered in the present article
arose from the entrywise matrix operations which preserve positivity,
and is highly relevant to them, its simplicity and versatility invites
the exploration of its potential utility beyond pure mathematics, to
other areas of much current interest. Compression, fast computation
and analysis of structured matrices are of great importance at
present. For example, large positive semidefinite matrices make an
important appearance in the analysis of big data, where they occur as
covariance or correlation matrices of random vectors. In such
settings, the rows and columns of the matrix are ordered according to
the ordering of the underlying variables. Entrywise operations such as
thresholding are natural in that setting (see \cite{Rothman2009}, for
example), as are grouping and averaging variables, and the related
block operations on the corresponding matrices. We conclude the paper
by { indicating} several such applications.

On the other hand, two notable connections with established areas of mathematics
should be mentioned. The partition \eqref{stratification} of $\mathbb{C}^{N \times N}$ into subsets $\mathcal{S}_\pi$
-- and more generally, as in Proposition \ref{Pstrata} -- is akin to the stratification of the flag variety in {\it Schubert calculus}.
It too emerges from a group action -- of $G^{N \times N}$ acting entrywise on $\mathbb{C}^{N \times N}$.

Second, the matrix compression and inflation operations we focus on are a very specific instance of a
\emph{conditional expectation} map. Originating in probability theory, conditional expectations were generalized 
to the non-commutative setting of operator algebras with hard to underestimate benefits.
From the ample bibliography on the subject we only mention an inspired, very recent survey \cite{Blecher}
which presents Blecher and Read's extension of positivity beyond the well-known $C^*$-algebraic setting.

In the present work, aimed at a large audience of practitioners of matrix analysis, we do not pursue in detail these two paths.

\subsection{Summary of contents}

In Section~\ref{Sstrat}, we study the stratification of complex
matrices with blocks that lie in single orbits under the action of a
multiplicative subgroup $G$ of the group of units~$\C^\times$, which
generalises the case $G = \{ 1 \}$ discussed above. In
Section~\ref{Spowers}, we show that the same stratification naturally
emerges from studying simultaneous entrywise powers or functions of
matrices that are positive semidefinite; in fact, we do not even
require the positivity of determinants of size $4$ or more.

In Section~\ref{Sentcalc}, we investigate the compression
$\down( A )$, which collapses every block of the matrix $A$ to a
single entry, equal to the average over that block. Compression and
its right inverse, inflation, provide an effective tool for operating
inside each stratum or its closure. Section~\ref{Sfuncalc} contains
results that describe how compression and inflation relate to spectra
and the functional calculus. These exploit the observation that a
weighted version of the compression map provides a $*$-isomorphism of
unital $C^*$~algebras between the stratum given by a partition with
$m$ elements and $\C^{m \times m}$.

The final section, Section~\ref{Srams}, is devoted to providing
examples and links to other areas of matrix analysis. Although our
study was prompted by the structure of matrices arising in the
statistics of big data, the isogenic stratification we introduce here
is relevant to sparse-matrix compression procedures and related fast
computational tools. Indeed, the averaging on isogenic blocks we
propose is a simple and efficient method of eliminating the redundancy
of operations involving this class of structured matrices. The rapidly
developing theory of \emph{hierarchical matrices}
\cite{Hackbusch-1999,BGH,Hackbusch-2016} is a natural framework where our
study finds deep resonances. A second setting where isogenic stratifications 
are very natural is that of \emph{coherent matrix organization}
\cite{GC,Gross,CBSW,MTCCK}. An immediate consequence of our technique is
the quick computation of spectra and singular {values} for
structured matrices; other possible applications are the efficient
solution of large linear systems and certifying the stability of
evolution semigroups.

\subsection{Acknowledgements}
The authors extend their thanks to the International Centre for
Mathematical Sciences, Edinburgh, where part of this work was carried
out, and to an anonymous referee, for their helpful comments.
D.G.~was partially supported by a University of Delaware Research
Foundation grant, by a Simons Foundation collaboration grant for
mathematicians, and by a University of Delaware Research Foundation
Strategic Initiative grant.  A.K.~was partially supported by Ramanujan
Fellowship SB/S2/RJN-121/2017, MATRICS grant MTR/2017/000295, and
SwarnaJayanti Fellowship grants SB/SJF/2019-20/14 and DST/SJF/MS/2019/3
from SERB and DST (Govt.~of India), by grant F.510/25/CAS-II/2018(SAP-I)
from UGC (Govt.~of India), by a Tata Trusts gravel grant, and by a Young
Investigator Award from the Infosys Foundation. M.P.~was partially
supported by a Simons Foundation collaboration grant.

\subsection{List of symbols}

We collect below some notation introduced and used throughout the
text.

\begin{itemize}
\item $\overline{D}( 0, \rho )$ is the closed disc in $\C$ with radius
$\rho$ centered at the origin, $S^1$ is the unit circle in $\C$ and
$\C^\times$ is the set of non-zero complex numbers; more generally,
the group of units for a unital commutative ring $R$ is denoted by
$R^\times$.

\item $\bJ_{M \times N}$ is the $M \times N$ matrix with each entry
equal to~$1$, whereas $\Id_N$ is the $N \times N$ identity matrix.

\item $f[ A ]$ is the matrix obtained by applying the
function~$f$ to each entry of the matrix~$A$.

\item $A^{\circ \alpha}$ is the matrix obtained from $A$ by taking the
$\alpha$th power of each entry, whenever this is well defined. We set
$0^0 := 1$.
  
\item $A^\dagger$ is the Moore--Penrose pseudo-inverse of~$A$.

\item $( \Pi_N, \prec )$ is the poset of partitions of
$\{ 1, \ldots, N \}$, where $\pi' \prec \pi$ if~$\pi$ is a
refinement of~$\pi'$.

\item $D_\pi$ is the $m \times m$ diagonal matrix
$\diag( | I_1 |, \ldots, | I_m | )$, for any partition
$\pi = \{ I_1, \ldots, I_m \} \in \Pi_N$.

\item $\down( A )$ is the $m \times m$ matrix obtained by averaging
the $N \times N$ matrix~$A$ over each block determined by the
$m$-element partition $\pi$.

\item $\ccdown( A )$ is the weighted $m \times m$ compression
$D_\pi^{1/2} \down( A ) D_\pi^{1/2}$ of the $N \times N$ matrix $A$.

\item $\up( B )$ is the $N \times N$ matrix given by inflating each
entry of the $m \times m$ matrix $B$ to a constant block subordinate
to the $m$-element partition~$\pi$.

\item $\ccup( B )$ is the weighted $N \times N$ inflation
$\up( D_\pi^{-1/2} B D_\pi^{-1/2} )$ of the $m \times m$ matrix $B$.

\item $\cP_N( \Omega )$ is the set of $N \times N$ positive
semidefinite matrices with entries in the set $\Omega$.
\end{itemize}

\section{Isogenic stratification of complex matrices}\label{Sstrat}

In order to define the isogenic stratification, we begin with a
summary of results appearing in our earlier work.

\begin{theorem}[{\cite[Theorems~5.1 and~5.8]{BGKP-fixeddim}}]%
\label{Tgroup}
Fix a multiplicative subgroup $G \subset \C^\times$, an integer
$N \geq 1$, and a non-zero matrix $A \in \cP_N( \C )$.
\begin{enumerate}
\item Suppose $\{ I_1, \ldots, I_m \}$ is a partition of
$\{ 1, \ldots, N \}$ satisfying the following two conditions.
\begin{enumerate}
\item Each diagonal block $A_{I_j}$ of $A$ is a submatrix having rank
at most one, and $A_{I_j} = \bu_j \bu_j^*$ for a unique
$\bu_j \in \C^{| I_j |}$ with first entry
$\bu_{j, 1} \in [ 0, \infty )$.

\item The entries of each diagonal block $A_{I_j}$ lie in a single
$G$-orbit.
\end{enumerate}
Then there exists a unique matrix $C = ( c_{i j})_{i,j = 1}^m$ such
that $c_{i j} = 0$ unless $\bu_i \neq 0$ and $\bu_j \neq 0$, and $A$
is a block matrix with
\[
A_{I_i \times I_j} = c_{i j} \bu_i \bu_j^* \qquad %
( 1 \leq i, j \leq m ).
\]
Moreover, the entries of each off-diagonal block of $A$ also
lie in a single $G$-orbit. Furthermore,
the matrix $C \in \cP_m( \overline{D}( 0, 1 ) )$, and
the matrices $A$ and $C$ have equal rank.

\item Consider the condition (c).
\begin{enumerate}
\item[(c)] The diagonal blocks of $A$ have maximal size, that is, each
diagonal block is not contained in a larger diagonal block that has
rank one.
\end{enumerate}
There exists a partition $\{ I_1, \ldots, I_m \}$ such that (a), (b)
and (c) hold, and such a partition is unique up to relabelling of the
indices.

\item Suppose (a)--(c) hold and $G = \C^\times$. Then the off-diagonal
entries of~$C$ lie in the open disc $D( 0, 1 )$.

\item If $G \subset S^1$, then blocks in a single $G$-orbit
automatically have rank at most one.
\end{enumerate}
\end{theorem}

Theorem~\ref{Tgroup} naturally leads to Schubert-cell type
stratifications of the cone $\cP_N( \C )$, and some properties of
these stratifications were studied in \cite{BGKP-fixeddim}. Here, we
explore a coarser form of partitioning on the whole set of $N \times
N$ complex matrices, based solely on $G$-equivariance.

\begin{definition}
We denote by $( \Pi_N, \prec )$ the poset of all \emph{partitions} of
the set $\{ 1, \ldots, N \}$, ordered so that $\pi' \prec \pi$ if and
only if $\pi$ is a refinement of $\pi'$.

The poset $\Pi_N$ is a lattice. Given $\pi$, $\pi' \in \Pi_N$, let
$\pi \wedge \pi'$ and $\pi \vee \pi'$ denote the meet and join of
$\pi$ and $\pi'$, respectively: see
\cite[Example~3.10.4]{Stanley-book-1}, but note that our ordering is
opposite to the one employed there. For example, the meet and join of
the partitions
\[
\pi = \{ \{ 1,2 \}, \{ 3,4 \}, \{ 5 \} \} %
\qquad \text{and} \qquad %
\pi' = \{ \{ 1,3,4 \}, \{ 2 \}, \{ 5 \} \}
\]
are
\[
\pi \wedge \pi' = \{ \{ 1, 2, 3, 4 \}, \{ 5 \} \} %
\qquad \text{and}
\qquad %
\pi \vee \pi' = \{ \{ 1 \}, \{ 2 \}, \{ 3, 4 \}, \{ 5 \} \}.
\]
The lattice $\Pi_N$ has maximum element
$\pi_\vee := \{ \{ 1 \}, \ldots, \{ N \} \}$ and minimum element
$\pi_\wedge := \{ \{ 1, \ldots, N \} \}$.
\end{definition}

\begin{theorem}\label{Texists}
Fix integers $M$, $N \geq 1$, a unital commutative ring $R$, and a
multiplicative subgroup~$G$ of units in $R^\times$. Given any matrix
$A \in R^{M \times N}$, there exist unique minimal (that is, coarsest) partitions
$\pi_{\min} = \{ I_1, \ldots, I_m \} \in \Pi_M$ and
$\varpi_{\min} = \{ J_1, \ldots, J_n \} \in \Pi_N$ such that the
entries of the block submatrix $A_{I_i \times J_j}$ lie in a single
$G$-orbit, for all $i \in \{ 1, \ldots, m \}$ and
$j \in \{ 1, \ldots, n \}$.

Now suppose $M = N$. If either $A$ is symmetric, or $R = \C$,
$G \subset S^1$, and $A$ is Hermitian,
then~$\pi_{\min} = \varpi_{\min}$.
\end{theorem}

\begin{proof}
The non-trivial part is to establish uniqueness. To do so, we claim
that if $( \pi_1, \varpi_1 )$ and
$( \pi_2, \varpi_2 ) \in \Pi_M \times \Pi_N$ satisfy the property in
the assertion, then so does
$( \pi_1 \wedge \pi_2, \varpi_1 \wedge \varpi_2 )$. The meet of
$\pi_1 \wedge \pi_2$ can be constructed as follows: connect $i$ and
$i' \in \{ 1, \ldots, M \}$ by an edge if they lie in the same block
of $\pi_1$ or $\pi_2$; this defines a graph with vertex set
$\{ 1, \ldots, M \}$ whose connected components yield the blocks of
the partition $\pi_1 \wedge \pi_2$. Denote this equivalence relation
by $i \sim i'$ in~$\pi_1 \wedge \pi_2$.

Now suppose $i \sim i'$ in $\pi_1 \wedge \pi_2$ and, similarly,
$j \sim j'$ in $\varpi_1 \wedge \varpi_2$, so there are paths joining
them, each of whose vertices lies in a block of $\pi_1$ or $\pi_2$,
and $\varpi_1$ or $\varpi_2$, respectively. Denote these paths by
\[
i = i_0 \leftrightarrow i_1 \leftrightarrow \cdots %
\leftrightarrow i_r = i' \quad \text{and} \quad %
j = j_0 \leftrightarrow j_1 \leftrightarrow \cdots %
\leftrightarrow j_s = j'.
\]
We claim that $a_{i j} \in G a_{i' j'}$. Indeed, using the above
paths,
\[
a_{i j} = a_{i_0 j} \in G a_{i_1 j} = G a_{i_2 j} = \cdots = %
G a_{i_r j} = G a_{i' j_0} = G a_{i' j_1} = \cdots = G a_{i' j'},
\]
and this proves the claim.

Now suppose $M = N$ and $A$ is symmetric. Then the partition
$( \pi_{\min}, \varpi_{\min} )$ works for $A$, and
$( \varpi_{\min}, \pi_{\min} )$ for $A^T = A$, whence the above analysis
shows $\pi_{\min} \wedge \varpi_{\min}$ works for both rows and columns
of~$A$. By minimality, it follows that
$\pi_{\min} = \pi_{\min} \wedge \varpi_{\min} = \varpi_{\min}$.

Finally, suppose $R = \C$, $G \subset S^1$, and $A$ is Hermitian. We
claim that $( \pi_{\min}, \varpi_{\min} )$ works for $A$ as well as
for $A^T = \overline{A}$; by the previous paragraph, this gives the
result. To show the claim, note that if $( \pi, \varpi )$ works
for~$A$ and~$G \subset S^1$, then $( \pi, \varpi )$ works for
$\overline{A}$, as~$G$ is closed under conjugation:
if $\zeta \in G \subset S^1$ then $a_{i j} = \zeta a_{i' j'}$ if and
only if
$\overline{a_{i j}} = \overline{\zeta} \overline{a_{i' j'}} = %
\zeta^{-1} a_{i' j'}$.
\end{proof}

Theorem~\ref{Texists} has a useful ``symmetric'' version for square matrices,
in the sense that the same partition is used for both the rows and the columns.
The proof is similar to that of Theorem~\ref{Texists} and is hence omitted.

\begin{proposition}\label{Pexists}
Fix an integer $N \geq 1$, and a multiplicative subgroup~$G$ of units
in a unital commutative ring $R$. Given $A \in R^{N \times N}$, there
exists a unique minimal partition
$\pi = \{ I_1, \ldots, I_m \} \in \Pi_N$, such that the entries of the
block submatrix $A_{I_i \times I_j}$ lie in a single $G$-orbit, for
all $i$, $j \in \{ 1, \ldots, N \}$.
\end{proposition}

The following definitions follow naturally from the previous
proposition. Our focus henceforth is on complex matrices with blocks
which are orbits of a fixed multiplicative subgroup $G$ of $S^1$, and
primarily the case $G = \{ 1 \}$, where the entries in any given block
are identical.

\begin{definition}\label{Dstratum}
Given a matrix $A \in \C^{N \times N}$ and a multiplicative group
$G \subset S^1$, let $\pi^G( A ) \in \Pi_N$ be the partition provided
by Proposition~\ref{Pexists} for the matrix $A$.

Conversely, for a given partition $\pi \in \Pi_N$, define the
\emph{stratum} $\stratumsymb^G_\pi$ to be
\[
\stratumsymb^G_\pi := \{ A \in \C^{N \times N} : \pi^G( A ) = \pi \}.
\]
\end{definition}

The proof of the following proposition is immediate.

\begin{proposition}\label{Pstrata}
Given $N \geq 1$ and a multiplicative subgroup $G$ of $S^1$, there is
a natural stratification of the set of $N \times N$ complex matrices
\[
\C^{N \times N} = \bigsqcup_{\pi \in \Pi_N} \stratumsymb^G_\pi,
\]
and the stratum $\stratumsymb^G_\pi$ has closure
\begin{equation}\label{Eschubert}
\overline{\stratumsymb^G_\pi} = %
\bigsqcup_{\pi' \prec \pi} \stratumsymb^G_{\pi'}
\end{equation}
when $\C^{N \times N}$ is equipped with its usual topology.
\end{proposition}

Equation~(\ref{Eschubert}) may be seen as akin to the Schubert-cell
decomposition of the flag variety corresponding to a semisimple Lie
group.

\begin{definition}
The family of stratifications given by Proposition~\ref{Pstrata} will
be referred to as \emph{$G$-block stratifications} of the
space~$\C^{N \times N}$.
\end{definition}

\begin{remark}\label{Rreasons}
There is an important distinction between the stratification of
$\C^{N \times N}$ considered here and that for the cone $\cP_N( \C )$
considered previously. In Theorem~\ref{Tgroup}, the partition was
defined to have the property that the diagonal blocks of
$A \in \cP_N( \C )$ have rank at most one. However, for a general
matrix~$A \in \C^{N \times N}$, this extra property need not hold for
$\pi^G( A )$, unless either $G = \{ 1 \}$, or~$A \in \cP_N( \C )$
and~$G \subset S^1$. In fact, as shown in
\cite[Proposition~4.6]{BGKP-pmp}, in the latter case the requirement
in Theorem~\ref{Tgroup} that $A$ is positive semidefinite may
be relaxed by requiring $A$ to be \emph{3-PMP}: every principal
minor of size no more than $3 \times 3$ is non-negative.
\end{remark}

\section{PMP matrices and simultaneous kernels of entrywise powers}\label{Spowers}

Henceforth, we will focus on constant-block stratifications, which we
call \emph{isogenic}, and we work mainly with square complex matrices.

We begin by noting how, for a large class of Hermitian matrices, the
partition $\pi^{\{ 1 \}}( A )$ introduced in Definition~\ref{Dstratum}
emerges naturally from the study of simultaneous kernels of entrywise
powers of $A$.

\begin{definition}[\cite{BGKP-pmp}]
Given an integer $k \geq 0$, a complex Hermitian matrix is said to be
\emph{$k$-PMP} (principal minor positive) if every $l \times l$
principal minor is non-negative, for all $l \in \{ 1, \ldots, k \}$.
\end{definition}

The $k$-PMP matrices interpolate between Hermitian matrices, the case
when $k = 0$, and positive semidefinite matrices, where $k$ is
maximal.

\begin{theorem}[{\cite[Theorem~5.1]{BGKP-pmp}}]\label{T3pmp}
Let the Hermitian matrix $A \in \C^{N \times N}$
be $3$-PMP, and let $\pi' = \{ I'_1, \ldots, I'_{m'} \}$ be any
partition refined by
$\pi = \pi^{\{ 1 \}}( A ) = \{ I_1, \ldots, I_m \}$. The following
spaces are equal.
\begin{enumerate}
\item The simultaneous kernel of $\one{N}$, $A$, \ldots,
$A^{\circ ( N - 1 )}$.

\item The simultaneous kernel of $A^{\circ n}$ for all $n \geq 0$.

\item The simultaneous kernel of the block-diagonal matrices
\[
\diag A_{\pi'}^{\circ n} :=
\bigoplus_{j = 1}^{m'} A^{\circ n}_{I'_j \times I'_j}
\]
for all $n \geq 0$.

\item The kernel of the matrix
$J_\pi := \bigoplus_{j = 1}^m \one{I_j}$.
\end{enumerate}
This equality of kernels need not hold for matrices that are not
$3$-PMP.
\end{theorem}

Our immediate goal is to show that a similar result holds for
arbitrary real powers.

\begin{proposition}\label{P3pmp}
Let the Hermitian matrix $A \in ( 0, \infty )^{N \times N}$ be
$3$-PMP, and let $\pi' = \{ I'_1, \ldots, I'_{m'} \}$ be any
partition refined by
$\pi = \pi^{\{ 1 \}}( A ) = \{ I_1, \ldots, I_m \}$. Fix real
powers $n_1 < \cdots < n_N$. The following spaces are equal.
\begin{enumerate}
\item The simultaneous kernel of $A^{\circ n_j}$ for $j = 1$, \ldots,
$N$.

\item The simultaneous kernel of $A^{\circ \alpha}$ for all
$\alpha \in \R$.

\item The simultaneous kernel of the block-diagonal matrices
\[
\diag A_{\pi'}^{\circ \alpha} := %
\bigoplus_{j = 1}^{m'} A^{\circ \alpha}_{I'_j \times I'_j}
\]
for all $\alpha \in \R$.

\item The kernel of the matrix
$J_\pi := \bigoplus_{j = 1}^m \one{I_j}$.
\end{enumerate}

The same holds when $A$ has non-negative entries with at least
one zero entry, as long as $n_1 = 0$ and $\alpha$ is taken to be
non-negative in (2) and (3).

These equalities of kernels need not hold for matrices that are
not $3$-PMP.
\end{proposition}

\begin{remark}
One subtlety in adapting the proof in \cite{BGKP-pmp} of
Theorem~\ref{T3pmp} to these variants is that a key matrix identity
relating $A^{\circ \alpha}$ to $A^{\circ n_1}$, \ldots,
$A^{\circ n_N}$ is no longer available in the required generality. In
\cite[Lemma~3.5]{BGKP-fixeddim}, it is shown that if $A$ is an
$N \times N$ matrix then
\[
A^{\circ M} = %
\sum_{j = 1}^N ( -1 )^{N - j} D_{M, j} A^{\circ ( j - 1 )} %
\qquad ( M \geq N ),
\]
where $D_{M, j}$ is a diagonal matrix composed of certain Schur
polynomials evaluated at the rows of $A$. If one tries to generalize
this identity from exponents $\{ 0, 1, \ldots, N - 1 \}$ to arbitrary
real powers, as in~\cite{KT}, the entries of the diagonal matrices
become ratios of generalized Vandermonde determinants, and such ratios
are not defined for all $A$. As a result, the above identity does not
admit a uniform generalization, so we cannot naively adapt the
previous proof to show that the subspace in~(2) contains that in~(1).
There is a similar issue when $\alpha$ is not greater than $n_N$,
which cannot be resolved by modifying the arguments in
\cite{BGKP-pmp}.
\end{remark}

In light of the preceding remark, it is heartening that all three
variants above follow from an even stronger result, that holds over
arbitrary subsets of $\C$ and for more general functions than
powers. To state this result, we introduce the following notion, over
an arbitrary commutative ring.

\begin{definition}\label{Dfulldetrank}
Let $X$ be a non-empty set and $R$ a unital commutative ring. A set
$\cF$ of functions from $X$ to $R$ has
\emph{full determinantal rank over~$X$} if, for any $k$ distinct
points $x_1$, \ldots, $x_k \in X$, where $k \leq | \cF |$,
there exist functions $f_1$, \ldots, $f_k \in \cF$ such that the
determinant $\det( f_i( x_j ) )$ is not a zero divisor.
\end{definition}

If $R$ is a field, $X$ is a finite set and $\cF$ has at least $| X |$
elements, then the family $\cF$ has full determinantal rank over $X$
if and only if the generalized matrix
$( f( x ) )_{x \in X, f \in \cF}$ has full rank. If, however, the
family $\cF$ has fewer than $| X |$ elements, then this is not the
case; for example, if $\cF$ contains a single function $f$ which is
the indicator function for a point in $X$, and $X$ contains at least
two points, then the matrix $( f( x ) )_{x \in X}$ has full rank but
zero entries{, and thus $\cF$ does not have full determinantal
rank over $X$.}

\begin{example}\label{Estrongli}
We now give several examples of such families of functions, the first
three of which correspond precisely to the results above.
\begin{enumerate}
\item If $X$ is a finite set of complex numbers with $k$ elements then
the family $\{ z^{j - 1} : j = 1, \ldots, k \}$ has full determinantal
rank over $X$. Thus $\{ z^n : n = 0, 1, 2, \ldots \}$ has full
determinantal rank over any subset of $\C$.

\item If $X$ is a finite set of positive real numbers with $k$
elements then $\{ x^{n_j} : j = 1, \ldots, k \}$ has full
determinantal rank over $X$ for any choice of distinct real exponents
$n_1$, \ldots, $n_k$. Thus $\{ x^\alpha : \alpha \in \R \}$ has full
determinantal rank over any subset of $( 0, \infty )$.

\item The family of functions $\{ x^\alpha : \alpha \geq 0 \}$ has
full determinantal rank over $[ 0, \infty )$.

\item The exponential family
$\{ \exp( \alpha x ) : \alpha \in \R \}$ has full determinantal rank
over $\R$. If $\alpha_1$, \ldots, $\alpha_k$ are distinct, and
similarly for $x_1$, \ldots, $x_k$, then the matrix
$( \exp( \alpha_i x_j ) )_{i, j = 1}^k$ is, up to transposition of
rows and columns, a generalized Vandermonde matrix and so
non-singular.

\item The Gaussian family
$\{ f_\alpha( x ) := \exp( -( x - \alpha )^2 ) { : \alpha
\in \R} \}$ has full determinantal rank over $\R$. If $\alpha_1$,
\ldots, $\alpha_k$ are distinct, and similarly for $x_1$, \ldots,
$x_k$, then the matrix
\[
( f_{\alpha_i}( x_j ) )_{i, j = 1}^k = %
\diag( e^{-\alpha_1^2}, \ldots, e^{-\alpha_k^2} ) M %
\diag( e^{-x_1^2}, \ldots, e^{-x_k^2} ),
\]
where $M = ( \exp( 2 \alpha_i x_j ) )$ is invertible by the preceding
example.
\end{enumerate}
\end{example}

Classes of functions with the separation property of
Definition~\ref{Dfulldetrank} are well studied in approximation
theory, under the name of \emph{Chebyshev systems}
\cite{Karlin-Studden}.

We now provide a theorem which contains all three results described
above.

\begin{theorem}\label{T3pmp-strong}
Fix a Hermitian matrix $A \in \C^{N \times N}$ that is $3$-PMP, where
$N \geq 1$. Suppose
$\pi := \pi^{\{ 1 \}}( A ) = \{ I_1, \ldots, I_m \}$ and
$\pi' = \{ I'_1, \ldots, I'_{m'} \}$ is any partition refined
by $\pi$. Let $X$ denote the set of entries of $A$ and suppose $\cF$
is a family of complex-valued functions on $X$ that has full
determinantal rank over the entries of each row of $A$. The following
spaces are equal.
\begin{enumerate}
\item The simultaneous kernel of $f[ A ]$ for all $f \in \cF$.

\item The simultaneous kernel of $f[ A ]$ for all functions
$f : X \to \C$.

\item The simultaneous kernel of the block-diagonal matrices
\[
f[ \diag A_{\pi'} ] := %
\bigoplus_{j = 1}^{m'} f[ A_{I'_j \times I'_j} ]
\]
for all functions $f : X \to \C$.

\item The kernel of $J_\pi := \bigoplus_{j = 1}^m \one{I_j}$.
\end{enumerate}
This equality of kernels need not hold for matrices that are not
$3$-PMP.
\end{theorem}

The proof of Theorem~\ref{T3pmp-strong} relies on the following
strengthening of a result obtained in the proof of
\cite[Theorem~5.1]{BGKP-pmp}; in that setting, the ring $R$ is taken
to be a field and $\cF = \{ 1, x, \ldots, x^{m - 1} \}$.

\begin{proposition}\label{P3pmp2}
Let $R$ be a unital commutative ring, and suppose the matrix
$B \in R^{m \times m}$ is such that $m \geq 1$ and
\begin{equation}\label{Erowcondition}
b_{i i} \neq b_{i j} \qquad \text{whenever } 1 \le i < j \le m.
\end{equation}
If the family $\cF$ has full determinantal rank over the entries in
each row of~$B$, then $\bigcap_{f \in \cF} \ker f[ B ] = \{ 0 \}$.
\end{proposition}

\begin{proof}
We show the result by induction on $m$, with the case $m = 1$ being
immediate. For the inductive step, we claim that if
$\bu \in \bigcap_{f \in \cF} \ker f[ B ]$ then $u_1 = 0$. This
reduces the problem to showing that the trailing principal
$( N - 1 ) \times ( N - 1 )$ submatrix of $B$ has the same
simultaneous kernel, whence we are done by the induction hypothesis;
note that if $\cF$ has full determinantal rank over a set $X$, then it
has so over any subset of $X$.

To show the claim, let $\br^T = ( r_1, \ldots, r_m )$ be the first row
of $B$, and apply Theorem~\ref{Texists} with $M = 1$, $N = m$ and
$G = \{ 1 \}$ to obtain a minimal partition
$\varpi_{\min} = \{ J_1, \ldots, J_k \}$ such that $\br^T$ is constant
on each block. Since $r_1 \neq r_j$ for $j = 2$, \ldots, $m$, we may
take $J_1 = \{ 1 \}$ without loss of generality. Let $\bs \in R^k$ be
the compression of $\br$ obtained by deleting repeated entries, so
that $s_j = r_l$ for any $l \in J_j$, with $s_1 = r_1 = b_{1 1}$.
As $\bs$ has distinct entries by construction, there exist $f_1$,
\ldots, $f_k \in \cF$ such that $\det C$ is not a zero divisor, where
$C := ( f_i( s_j ) )_{i, j = 1}^k$.

If $\bu \in \bigcap_{f \in \cF} \ker f[ B ]$, then
$f_i[ \br ]^T \bu = 0$ for $i = 1$, \ldots, $k$. Let $\bv \in R^k$ be
defined by setting $v_j := \sum_{l \in J_j} u_l$, and note that
$v_1 = u_1$. It follows that $f_i[ \bs ]^T \bv = 0$ for $i = 1$,
\ldots, $k$, so that $C \bv = 0$. By Cramer's rule, it follows that
$\det( C ) \bv = 0$ in $R^k$, whence $\bv = 0$ by the hypotheses. In
particular, we have that $u_1 = v_1 = 0$, as desired.
\end{proof}

We can now prove the general theorem described above.

\begin{proof}[Proof of Theorem~\ref{T3pmp-strong}]
Let $V_1$, \ldots, $V_4$ be the subspaces described in parts (1) to
(4) of the statement of the theorem. We will show a chain of
inclusions.

Note first that $\bv \in V_4$ if and only if
$\sum_{l \in I_j} v_l = 0$ for $j = 1$, \ldots, $m$, from which it
follows that $V_4 \subset V_2 \cap V_3$. Furthermore, it is immediate
that $V_2 \subset V_1$.
We now claim that the inclusion $V_1 \subset V_4$ gives the
result. Firstly, we then have that
\[
V_4 \subset V_2 \cap V_3 \subset V_2 \subset V_1 \subset V_4
\]
and secondly, this also gives the inclusion $V_3 \subset V_4$, in
which case
\[
V_4 \subset V_2 \cap V_3 \subset V_3 \subset V_4.
\]
For the last claim, note that $V_3$ is the direct sum of
$V'_j := \bigcap_f \ker f[ A_{I'_j \times I'_j} ]$ for $j = 1$,
\ldots, $m'$, where the intersection is taken over the set of all
functions from $X$ to $\C$. Each $V'_j$ is contained in the
simultaneous kernel of $f[ A_{I'_j \times I'_j} ]$ with $f \in \cF$,
so the inclusion $V_1 \subset V_4$ applied to each matrix
$A_{I'_j \times I'_j}$ with the partition $\pi \cap I'_j$ gives the
inclusion as required.

Thus, it remains to show that $V_1 \subset V_4$. We proceed as in
\cite[Proof of Theorem~5.1]{BGKP-pmp}. Let $B \in \C^{m \times m}$ be
the \textit{compression} of $A$, denoted $B = \down( A )$, so that the
$( i, j )$ entry of $B$ equals the value $A$ takes on the block
$I_i \times I_j$. Note that $B$ inherits the $3$-PMP property from
$A$.

Without loss of generality, we suppose that
$b_{1 1} \geq b_{22} \geq \cdots \geq b_{m m} \geq 0$. It must hold
that $b_{i i} > 0$ for $i = 1$, \ldots, $m - 1$. If not, then
$b_{j j} = 0$ for $j = i$, \ldots, $m$, whence all the entries of $B$
in the $\{ i, \ldots, m \} \times \{ i, \ldots, m \}$ block are zero,
by the $2$-PMP property, which contradicts the minimality of
$\pi = \pi^{\{ 1 \}}( A )$.

We next claim that $B$ satisfies condition (\ref{Erowcondition}) of
Proposition~\ref{P3pmp2}. If there exist $i$,
$j \in \{ 1, \ldots, m \}$ with $i < j$ and $b_{i j} = b_{i i}$, then
\[
b_{i i}^2 \geq b_{i i} b_{j j} \geq | b_{i j} |^2 = b_{i i}^2,
\]
so $b_{j j} = b_{i i}$ because $b_{i i} > 0$. Hence $B$ is constant on
the block $\{ i, j \}$, which again violates the definition of $\pi$.

We can now conclude our proof that $V_1 \subset V_4$. If
$\bu \in \C^N$ and $\bv \in \C^m$ is defined by setting
$v_j := \sum_{i \in I_j} u_i$, then $\bv$ is the compression of
$J_\pi \bu$, whence $\bu \in \ker J_\pi$ if and only if $\bv = 0$. Now
suppose $\bu \in V_1 = \bigcap_{f \in \cF} \ker f[ A ]$. Then
$\bv \in \bigcap_{f \in \cF} \ker f[ B ] = \{ 0 \}$, by 
Proposition~\ref{P3pmp2}, so
$\bu \in \ker J_\pi = V_4$ as required. This concludes the proof of
the chain of inclusions.

The counterexample is as in \cite{BGKP-pmp}: let
$N = 3 k + 2$ for some $k \geq 1$, let $A$ be the Toeplitz matrix with
$( i, j )$ entry $1$ if $| i - j | \leq 1$ and $0$ otherwise, and note
that $A$ is $2$-PMP but not $3$-PMP. It is immediate that the family
$\{ f_1( x ) = 1, f_2( x ) = x \}$ has full determinantal rank on the
set $\{ 0, 1 \}$ of entries of $A$ and the vector
$( 1, -1, 0, 1, -1, 0, \cdots, 1, -1 )^T$ lies in
$\ker f_1[ A ] \cap \ker f_2[ A ]$. However, as
$\pi^{\{ 1 \}}( A ) = \{ \{ 1 \}, \ldots, \{ N \} \}$, so
$J_\pi = \Id_N$ and $\ker J_\pi = \{ 0 \}$.
\end{proof}

\section{Inflation and compression}\label{Sentcalc}

In the present section we continue to explore isogenic
stratification. The proof of Theorem~\ref{T3pmp-strong} used the
\emph{compression operator} $\down$, which will be one of the main
characters in the new act.

Throughout this section, we fix a partition
$\pi = \{ I_1, \ldots, I_m \} \in \Pi_N$, where $N \geq 1$.

Suppose $A$, $B \in \overline{\stratum{\pi}^{\{ 1 \}} }$, so that these
matrices are constant on the blocks defined by the partition $\pi$.
Then we may write
\[
A = \sum_{i, j = 1}^m a_{i j} \bJ_{I_i \times I_j} %
\qquad \text{and} \qquad %
B = \sum_{i, j = 1}^m  b_{i j} \bJ_{I_i \times I_j},
\]
where $\bJ_{I_i \times I_j}$ is the $N \times N$ matrix with $1$ in
each entry of the $I_1 \times I_j$ block and $0$ elsewhere. Hence
\[
A B = %
\sum_{i, j, k = 1}^m a_{i k} b_{k j} | I_k | \bJ_{I_i \times I_j} %
\qquad \text{and} \qquad %
A \circ B = \sum_{i, j = 1}^m a_{i j} b_{i j} \bJ_{I_i \times I_j},
\]
so the closure of every stratum is a subalgebra of $\C^{N \times N}$
for both the usual and entrywise multiplication. The isogenic
stratification of the space $\C^{N \times N}$ is not merely into
linear spaces of matrices, but into subalgebras.

Next we focus on the compression operation of a fixed stratum 
to a lower-dimensional space. To simplify notation, we
write henceforth
\[
\stratum{\pi} = \stratum{\pi}^{\{ 1 \}}
\]
whenever there is no danger of confusion.

\begin{definition}
For $i$, $j \in \{ 1, \ldots, m \}$, let $E_{i j}$ denote the
elementary matrix with $( i, j )$ entry equal to $1$ and all other
entries $0$, and recall that $\bJ_{I_i \times I_j}$ is the
$N \times N$ matrix with $1$ in each entry of the $I_i \times I_j$
block and $0$ elsewhere.

\begin{enumerate}
\item Define the linear \emph{inflation} map {as the linear extension of}
\[
\up : \C^{m \times m} \to \C^{N \times N}; \ %
E_{i j} \mapsto \bJ_{I_i \times I_j}
\]
and note that the range of $\up$ is $\overline{\stratum{\pi}}$.

\item Define the linear \emph{compression} map
\[
\down : \C^{N \times N} \to \C^{m \times m}; \ %
\down( A )_{i j} := %
\frac{1}{| I_i | \, | I_j |} \sum_{p \in I_i, q \in I_j} a_{p q},
\]
so that the image $B = \down( A )$ is such that $b_{i j}$ is the
arithmetic mean of the entries in $A_{I_i \times I_j}$, for
$i$, $j = 1$, \ldots, $N$.
\end{enumerate}
\end{definition}

Our next result shows that we may, in the entrywise setting, compress
the matrices in a given stratum and work with the resulting smaller
matrices with no loss of information.

\begin{theorem}\label{Tentrywise}
Let $\overline{\stratum{\pi}}$ and $\C^{m \times m}$ be equipped with
the entrywise product, so that the units for this product are
$\one{N}$ and $\one{m}$, respectively. The maps
\[
\down : \overline{\stratum{\pi}} \to \C^{m \times m} %
\qquad  \text{and} \qquad %
\up : \C^{m \times m} \to \overline{\stratum{\pi}}
\]
are mutually inverse, rank-preserving isomorphisms of unital
$*$-algebras. Moreover, a matrix $A \in \overline{\stratum{\pi}}$ is
positive semidefinite if and only if~$\down( A )$ is.
\end{theorem}

Towards the proof of this result, we first study the inflation and
compression operators. To this aim, some new terminology will be
useful.

\begin{definition}\hfill
\begin{enumerate}
\item Define the \emph{weight} matrix $\cW_\pi \in \C^{N \times m}$ to
have $(i,j)$ entry~$1$ if $i \in I_j$ and~$0$ otherwise. Let
\begin{equation}
D_\pi := \cW_\pi^* \cW_\pi = \diag( | I_1 |, \ldots, | I_m | ).
\end{equation}
When the rows of $\cW_\pi$ are ordered so that the indices
in~$I_1$ are first, then the indices in $I_2$, and so on, then
\[
\cW_\pi \cW_\pi^* = %
\diag( \bJ_{| I_1 |}, \ldots, \bJ_{| I_m |} ) =: J_\pi.
\]

\item For any coarsening $\pi' \prec \pi$, define the partition
$\pi'_\downarrow \in \Pi_m$ so that the blocks of $\pi'_\downarrow$
are made up of those indices of blocks in $\pi$ to be combined to form
the blocks of~$\pi'$. Thus if
\[
\pi = \{ I_1 = \{ 1, 2 \}, I_2 = \{ 3, 4 \}, I_3 = \{ 5 \}, %
I_4 = \{ 6, 7, 8 \} \}
\]
and
\[
\pi' = \{ \{ 1, 2 \}, \{ 3, 4, 5 \}, \{ 6, 7, 8 \} \},
\]
then $\pi'_\downarrow = \{ \{ 1 \}, \{ 2, 3 \}, \{ 4 \} \}$. Denote
the inverse map from $\Pi_m$ to~$\Pi_N$ by
$\pi'' \mapsto \pi''_\uparrow$.
\end{enumerate}
\end{definition}

Several important properties of the operators $\down$ and~$\up$ are
summarized in Proposition~\ref{Pdownup}.

\begin{proposition}\label{Pdownup}\hfill
\begin{enumerate}
\item The map $\pi' \mapsto \pi'_\downarrow$ is a bijection between
the set of all coarsenings of~$\pi$ in $\Pi_N$ and the set $\Pi_m$.

\item For all $A \in \C^{m \times m}$, 
\begin{equation}\label{Eup}
\up( A ) = \cW_\pi A \cW_\pi^* \in \overline{\stratum{\pi}}.
\end{equation}
Moreover, $\up$ is a bijection from $\C^{m \times m}$ onto
$\overline{\stratum{\pi}}$, sending the stratum
$\stratum{\pi'} \subset \C^{m \times m}$ to the stratum
$\stratum{\pi'_\uparrow} \subset \overline{\stratum{\pi}}$,
for any $\pi' \in \Pi_m$.

\item For all $A \in \C^{N \times N}$, 
\begin{equation}\label{Edown}
\down( A ) = %
D_\pi^{-1} \cW_\pi^* A \cW_\pi D_\pi^{-1} \in \C^{m \times m}.
\end{equation}
Moreover, $\down$ restricted to $\overline{\stratum{\pi}}$ is a
bijection onto $\C^{m \times m}$, being the inverse map of $\up$.

\item The linear maps $\up$ and $\down$ are compatible with matrix
multiplication in the following sense:
\begin{equation}\label{Emultup}
\up( A B ) = \up( A D_\pi^{-1 / 2} ) \up( D_\pi^{-1 / 2 } B ) %
\quad \text{for all } A, B \in \C^{m \times m}
\end{equation}
and
\begin{equation}\label{Emultdown}
\down( A B ) = \down( A ) D_\pi \down( B ) \qquad %
\text{for all } A, B \in \overline{\stratum{\pi}}.
\end{equation}
\end{enumerate}
\end{proposition}

\begin{proof}
Part (1) readily follows from the definitions. To see that (\ref{Eup})
holds, it suffices by linearity to show that
$\cW_\pi E_{i j} \cW_\pi^* = \bJ_{I_i \times I_j}$ for each elementary
matrix $E_{i j}$, but this is immediate. It is also clear that~$\down$
and~$\up$ are mutually inverse bijections between
$\overline{\stratum{\pi}}$ and $\C^{m \times m}$. The other assertions
of (2) are straightforward.

To show that (\ref{Edown}) holds, it once again suffices to take~$A$
to be an arbitrary elementary matrix; the calculation is then
straightforward. That $\down$ and $\up$ are inverse between
$\overline{\stratum{\pi}}$ and $\C^{m \times m}$ has already been
discussed.

Finally, equation~(\ref{Emultup}) is verified by using (\ref{Eup}) and
the definition of $D_\pi$. To see (\ref{Emultdown}), note that, by
(\ref{Emultup}),
\begin{align*}
\up( \down( A ) D_\pi \down( B ) ) & = %
\up( \down( A ) D_\pi^{1 / 2} D_\pi^{1 / 2} \down( B ) ) \\
 & = \up( \down( A ) ) \up( \down( B ) ) \\
 & = A B,
\end{align*}
and this concludes the proof.
\end{proof}

In other words, the map $\up \comp \down$ is a conditional expectation
on $\C^{N \times N}$ corresponding to the $\sigma$-algebra generated
by the blocks which define the stratum~$\stratum{\pi}$.

These properties of inflation and compression maps help demonstrate
the $*$-isomorphism claimed above.

\begin{proof}[Proof of Theorem~\ref{Tentrywise}]
By Proposition~\ref{Pdownup}, it suffices to show the results only for
$\up$. This map is linear and multiplicative for the entrywise
product, by definition. Equation (\ref{Eup}) shows that $\up$ commutes
with taking the adjoint~$*$, and also that $\up$ preserves rank,
since~$\cW_\pi$ has full rank. Finally, that positivity is preserved
follows immediately from (\ref{Eup}).
\end{proof}

\section{Spectral permanence}\label{Sfuncalc}

Theorem \ref{Tentrywise} shows that the map
$\down : \overline{\stratum{\pi}} \to \C^{m \times m}$ is a
positivity-preserving $*$-algebra isomorphism for the entrywise
product, whence the entrywise calculus is transported to a
lower-dimensional space of matrices. However, it is not immediately
apparent if the usual holomorphic functional calculus in
$\C^{m \times m}$ can be transported up to the closure of each
stratum. We now show how this can be accomplished with the help of
different weighted inflation and compression maps.

As in the previous Section~\ref{Sentcalc}, we fix a partition
$\pi = \{ I_1, \ldots, I_m \} \in \Pi_N$, where $N \geq 1$.

\begin{definition}
Define linear operators
\[
\ccdown : \C^{N \times N} \to \C^{m \times m} %
\qquad \text{and} \qquad %
\ccup : \C^{m \times m} \to \C^{N \times N}
\]
by setting
\[
\ccdown( A ) := D_\pi^{1 / 2} \down( A ) D_\pi^{1 / 2} = %
D_\pi^{-1 / 2} \cW_\pi^* A \cW_\pi D_\pi^{-1 / 2}
\]
and
\[
\ccup( B ) := \up( D_\pi^{-1 / 2} B D_\pi^{-1 / 2} ) = %
\cW_\pi D_\pi^{-1 / 2} B D_\pi^{-1 / 2} \cW_\pi^*.
\]
\end{definition}

\begin{theorem}\label{Tfunctional}
The maps $\ccdown$ and $\ccup$ are mutually inverse, rank-preserving
isomorphisms between the unital $*$-algebras
$\overline{\stratum{\pi}}$ and $\C^{m \times m}$ equipped with the
usual matrix multiplication. Furthermore, a matrix
$A \in \overline{\stratum{\pi}}$ is positive semidefinite if and only
if $\ccdown(A)$ is.
\end{theorem}

\begin{proof}
That $\ccdown$ and $\ccup$ are linear, bijective, $*$-equivariant, and
preserve rank follows from the corresponding properties of $\down$ and
$\up$ from Theorem~\ref{Tentrywise}, since $D_\pi$ is positive
definite. Moreover, it is easily shown that $\ccdown$ and $\ccup$ are
mutually inverse between $\overline{\stratum{\pi}}$ and
$\C^{m \times m}$, since $\down$ and $\up$ are. Furthermore, if
$A$, $B \in \overline{\stratum{\pi}}$, then, by~(\ref{Emultdown}),
\[
\ccdown( A B ) = D_\pi^{1/2} \down( A B ) D_\pi^{1/2} = %
D_\pi^{1/2} \down( A ) D_\pi \down( B ) D_\pi^{1/2} = %
\ccdown( A ) \ccdown( B ).
\]
Thus the maps $\ccdown$ and $\ccup$ are algebra homomorphisms for the
usual matrix multiplication, and that they take the units to one
another is readily verified; note that $\overline{\stratum{\pi}}$ has
unit
\[
1_{\overline{\stratum{\pi} }} = \ccup( 1_{\C^{m \times m}} ) = %
\sum_{k = 1}^m | I_k |^{-1} \one{I_k}.
\]
The final assertion about preserving positivity is immediate.
\end{proof}

\begin{remark}
We note a couple of simple consequences of Theorem~\ref{Tfunctional}.
\begin{enumerate}
\item The matrix $\down( A )$ is invertible if and only if
$\ccdown( A )$ is, in which case
\[
\ccdown( A )^{-1} = D_\pi^{-1 / 2} \down( A )^{-1} D_\pi^{-1 / 2}.
\]

\item If $A \in \overline{\stratum{\pi}}$, then the Moore--Penrose
pseudo-inverse is:
\begin{equation}\label{Emoorepenrose}
A^\dagger = \ccup( \ccdown( A )^\dagger ),
\end{equation}
since $A = \ccup( \ccdown( A ) )$ and $\ccup$ is a $*$-algebra
homomorphism.
\end{enumerate}
\end{remark}

We now extend the compression and inflation operators to act on
vectors as well as matrices.

\begin{definition}\label{Dupdownvec}
Define
\[
\down : \C^N \to \C^m \qquad \text{and} \qquad %
\ccdown : \C^N \to \C^m
\]
by setting
\[
\down( \bu )_j = | I_j |^{-1} \sum_{k \in I_j} u_k %
\qquad \text{and} \qquad %
\ccdown( \bu )_j = | I_j |^{-1 / 2} \sum_{k \in I_j} u_k.
\]
Similarly, define
\[
\up : \C^m \to \C^N \qquad \text{and} \qquad %
\ccup : \C^m \to \C^N
\]
by setting
\[
\up( \bv )_j =  v_k  \qquad \text{and} \qquad %
\ccup( \bv )_j = | I_k |^{-1 / 2} v_k  \quad ( j \in I_k ).
\]
\end{definition}

The following proposition summarizes basic properties of the operators
$\down$, $\ccdown$,$\up$, and $\ccup$ acting on vectors. Its proof is
omitted.

\begin{proposition}\label{Pupdownvec}
Let $\bu \in \C^N$, $\bv$, $\bw \in \C^m$,
$A \in \overline{\stratum{\pi}}$, and $B \in \C^{m \times m}$.
\begin{enumerate}
\item $\ccdown( \bu ) = D_\pi^{1 / 2} \down( \bu )$ and
$\ccup( \bv ) = \up( D_\pi^{-1 / 2} \bv )$.
  
\item $\ccdown( \ccup( \bv ) ) = \bv = \down( \up( \bv ) )$.

\item $\ccup( \bv )^* \ccup( \bw ) = \bv^* \bw$ and
$\up( \bv )^* \up( \bw ) = \bv^* D_\pi \bw$.

\item $\ccdown( A \bu ) = \ccdown( A ) \ccdown( \bu )$ and
$\down( A \bu ) = \down( A ) D_\pi \down( \bu )$.

\item $\ccup( B \bv ) = \ccup( B ) \ccup( \bv )$ and
$\up( B \bv ) = \up( B ) \up( D_\pi^{-1} \bv )$.
\end{enumerate}
\end{proposition}

\begin{remark}\label{Rupdown}
We collect here some further properties of the maps $\down$, $\up$,
$\ccdown$, and $\ccup$.

\begin{enumerate}
\item If $A \in \overline{\stratum{\pi}}$ and $\bu \in \C^N$, then
$A \bu$ is constant on the blocks of the partition $\pi$, that is,
$A \bu \in \im \up = \im \ccup$, the range of these two operators
acting from $\C^m$ to $\C^N$. Consequently,
\[
A \bu = \ccup( \ccdown(A) \ccdown( \bu ) ) = %
\up( \down( A ) D_\pi \down(\bu) ).
\]

\item If $\bu$, $\bv \in \im \up = \im \ccup$ and
$A \in \overline{\stratum{\pi}}$, then
\[
\bu^* \bv = \ccdown( \bu )^* \ccdown( \bv ) %
\qquad \text{and} \qquad %
\bu^* A \bv = \ccdown( \bu )^* \ccdown( A ) \ccdown ( \bv ),
\]
whereas
\[
\bu^* \bv = \down( \bu )^* D_\pi \down( \bv ) %
\quad \text{and} \quad %
\bu^* A \bv = \down( \bu )^* D_\pi \down( A ) D_\pi \down ( \bv ).
\]
Furthermore, if $\bx$, $\by \in \C^m$ then
\[
\ccup( \bx \by^* ) = \ccup( \bx ) \ccup( \by )^* %
\quad \text{and} \quad %
\up( \bx \by^* ) = \up( \bx ) \up( \by )^*,
\]
so
\[
\ccdown( \bu \bv^* ) = \ccdown( \bu ) \ccdown( \bv )^* %
\quad \text{and} \quad %
\down( \bu \bv^* ) = \down( \bu ) \down( \bv )^*.
\]
\end{enumerate}
\end{remark}

Our next result shows that the maps $\ccup$ and $\ccdown$ preserve
eigenvalues, up to the possible addition or removal of~$0$.

\begin{proposition}\label{Peigen}
The following spectral permanence holds for all
$A \in \overline{\stratum{\pi}}$ and~$B \in \C^{m \times m}$:
\begin{align}
\sigma( A ) \setminus \{ 0 \} & = %
\sigma( \ccdown( A ) ) \setminus \{ 0 \} \label{Edownspecev} \\
\text{and} \quad \sigma( B )\setminus \{ 0 \} & = %
\sigma( \ccup( B ) ) \setminus \{ 0 \} \label{Eupspecev}.
\end{align}
In particular,
$\ccdown : \overline{\stratum{\pi}} \to \C^{m \times m}$
and $\ccup : \C^{m \times m} \to \overline{\stratum{\pi}}$ preserve
the spectral radius. Furthermore, applying $\ccup$ to $B$ adds a zero
eigenvalue of geometric multiplicity $N - m$ to $\sigma( B )$, whereas
applying $\ccdown$ to $A$ reduces the geometric multiplicity of the
zero eigenvalue of $A$ by $N - m$.
\end{proposition}
\begin{proof}
To prove (\ref{Eupspecev}), first let $B \bv = \lambda \bv$ for some
$\lambda \neq 0$ and $\bv \in \C^m \setminus \{ 0 \}$. Then, by
Proposition~\ref{Pupdownvec}(5),
\[
\ccup( B ) \ccup( \bv ) = \ccup( B \bv ) = \lambda \ccup( \bv )
\]
and $\ccup( \bv ) \neq 0$, so
$\lambda \in \sigma( \ccup( B ) ) \setminus \{ 0 \}$. Conversely,
suppose $\ccup( B ) \bu = \lambda \bu$ for some $\lambda \neq 0$ and
$\bu \in \C^N \setminus \{ 0 \}$. Then
$\bu = \lambda^{-1} \ccup( B ) \bu\in \im \ccup$, and therefore
$\bu = \ccup( \bv )$ for some $\bv \in \C^m \setminus \{ 0 \}$.
Hence, using Proposition~\ref{Pupdownvec}(5) again,
\[
\ccup( B \bv - \lambda \bv ) = %
\ccup( B ) \ccup( \bv ) - \lambda \ccup( \bv ) = 0.
\]
It follows that $B \bv - \lambda \bv = 0$, and this proves the first
equality.

The proof of (\ref{Edownspecev}) is similar, with the help of
Proposition~\ref{Pupdownvec} and that the fact that if
$\bu \in \im \ccup$ and $\ccdown( \bu ) = 0$ then $\bu = 0$.

The last claim follows immediately by the rank-nullity theorem.
\end{proof}

\begin{remark}
Let $\alg$ be a complex associative algebra with multiplicative
identity $1_\alg$. {Then the \emph{spectrum} of $a \in \alg$ is defined as} 
\[
\spec( a; \alg ) := \{ \lambda \in \C : %
\lambda 1_\alg - a \text{ is not invertible in } \alg \}.
\]
In the notation of Proposition~\ref{Peigen},
$\sigma( A ) = \sigma( A ; \C^{N \times N} )$ for any
$A \in \C^{N \times N}$, where $\C^{N \times N}$ equipped with the
usual matrix product.

Since $\ccdown$ is an algebra isomorphism from
$\overline{\stratum{\pi}}$ to $\C^{m \times m}$, and the spectrum and
the set of eigenvalues coincide in the latter, it follows from
(\ref{Eupspecev}) and (\ref{Edownspecev}) that, for any
$A \in \overline{\stratum{\pi}}$,
\begin{equation}\label{Especdown1}
\spec( A; \overline{\stratum{\pi}} ) = %
\spec( \ccdown( A ); \C^{m \times m} ) = %
\spec( \ccdown( A ) ) \subset \spec( A )
\end{equation}
and
\begin{equation}\label{Especdown2}
\spec( A ) \setminus \{ 0 \} \subset %
\spec( \ccdown( A ) ) = \spec( \ccdown( A ); \C^{m \times m} ) = %
\sigma( A; \overline{\stratum{\pi}} ).
\end{equation}
The matrix
$A = \left( \begin{smallmatrix} 1 & 1 \\
 1 & 1 \end{smallmatrix} \right) \in \stratum{\pi_\wedge}$
has $\spec( A; \overline{\stratum{\pi_\wedge}} ) = \{ 2 \}$ and
$\spec( A ) = \{ 0, 2 \}$, so the inclusion in (\ref{Especdown1}) may
be strict. Similarly, the matrix
$A = \left( \begin{smallmatrix} 1 & 0 \\
 0 & 0 \end{smallmatrix} \right) \in \stratum{\pi_\vee}$
has
$\spec( A; \overline{\stratum{\pi_\vee}} ) = \spec( A ) = \{ 0, 1 \}$,
so the inclusion in~(\ref{Especdown2}) may also be strict.
\end{remark}

The following theorem shows that the holomorphic functional calculus
naturally transfers between $\C^{m \times m}$ and
$\overline{\stratum{\pi}}$. Its proof follows immediately from the
fact that $\ccdown$ is an algebra isomorphism.

\begin{theorem}
Given $A \in \overline{\stratum{\pi}}$ and $B \in \C^{m \times m}$,
let the resolvents
\[
R( z; A ) := ( z J_\pi - A )^{-1} \quad \text{for all } %
z \in \spec( A; \overline{\stratum{\pi}} )
\]
and
\[
R( z; B ) := ( z I_{m \times m} - B )^{-1} \quad %
\text{for all } z \in \spec( B; \C^{m \times m} ) = \spec( B ).
\]
If
$z \in \spec( A; \overline{\stratum{\pi}} ) = %
\spec( \ccdown( A ); \C^{m \times m} ) = \spec( \ccdown( A ) )$, then
\[
\ccdown( R( z; A ) ) = R( z; \ccdown( A ) ).
\]
Thus, the holomorphic functional calculus transfers between
$\C^{m \times m}$ and $\overline{\stratum{\pi}}$: if
$A \in \overline{\stratum{\pi}}$ and $f$ is holomorphic on an open set
containing $\spec( A; \overline{\stratum{\pi}} )$, then
\[
f( A ) \in \overline{\stratum{\pi}} \qquad \text{and} \quad 
\ccdown( f( A ) ) = f( \ccdown( A ) ).
\]
\end{theorem}

\begin{remark}
{It} follows from Theorem~\ref{Tfunctional} that analogues of the
general linear group $GL_m$, unitary group $U_m$, and permutation
group $S_m$ exist inside the stratum
$\overline{\stratum{\pi}}$. Furthermore, the notions of nilpotent,
Hermitian, and positive semidefinite matrices are preserved in
$\overline{\stratum{\pi}}$ via $\ccup$. Hence analogues of the Bruhat,
Cholesky, and polar decompositions can also be defined
on~$\overline{\stratum{\pi}}$, and all of the respective factors live
inside the same stratum.
\end{remark}

We conclude this section with some remarks on the situation for
$\stratum{\pi}^G$ with more general $G \subset \C^\times$. A key
feature in the definition of the compression operator~$\down$, and so
of $\ccdown$, was the unique decomposition of a single rank-one block
matrix: $\bJ_{I_i \times I_j} = \bJ_{I_i} \bJ_{I_j}^*$. When
$G \neq \{ 1 \}$, the block matrices do not possess such a
decomposition. Moreover, each stratum and its closure are no longer
closed under multiplication. For example, if $G = \{ \pm1 \}$,
$N = 2$, and $\pi = \{ \{ 1, 2 \} \}$ is the minimum partition, then
\[
A = \begin{pmatrix} \phantom{-}1 & -1 \\
\phantom{-}1 & \phantom{-}1 \end{pmatrix} \in \stratumsymb^G_\pi, %
\quad \text{but} \quad A^2 = %
\begin{pmatrix} \phantom{-}0 & -2 \\
\phantom{-}2 & \phantom{-}0 \end{pmatrix}
\not\in \stratumsymb^G_\pi.
\]

\section{Ramifications}\label{Srams}

We collect in this last section several observations revealing natural
links between the compression and inflation operations derived from
the isogenic stratification of the matrix space and recent advances or
classical examples of current interest in numerical matrix analysis.

\subsection{Symmetric statistical models}

We briefly discuss a related setting of statistical models and
covariance matrices which exhibit symmetry with respect to a group of
permutations; see \cite{SC-groupsym} and the references therein. The
authors fix there a subgroup $\fG$ of the symmetric group $S_N$, and
define
\[
\wg := \{ A \in \R^{N \times N} : a_{i j} = a_{\sigma(i), \sigma(j)} %
\text{ whenever } 1 \leq i, j \leq N \text{ and } \sigma \in \fG \}.
\]

In the framework of the present article we consider permutation groups
associated to a partition $\pi = \{ I_1, \ldots, I_m \}$ of
$\{ 1, \ldots, N \}$, that is,
\[
\fG_\pi := \Aut( I_1 ) \times \cdots \times \Aut( I_m ).
\]
By contrast, in \cite{SC-groupsym} the authors work with more general
subgroups of $S_N$. Furthemore, the matrices in a given stratum
$\stratum{\pi}^{\{1\}}$ have diagonal blocks with equal entries, as
the permutations act separately on the rows and columns of
$\C^{N \times N}$, whereas the diagonal blocks of a square matrix in
$\wg$ may have different diagonal and off-diagonal entries.

Having acknowledged these differences, we now discuss the setting
of~\cite{SC-groupsym} from the viewpoint adopted above. The first step
is to establish the existence of a suitable partition associated with
a given matrix.

\begin{proposition}\label{Pgroupsym}
Fix a unital commutative ring $R$ and a multiplicative subgroup
$G \subset R^\times$. Given a matrix $A \in R^{N \times N}$, where
$N \geq 1$, there is a unique minimal partition
$\varpi_{\min} = \{ I_1, \ldots, I_m \} \in \Pi_N$ satisfying the
following two properties.
\begin{enumerate}
\item If $i$, $j \in \{ 1, \ldots, N \}$ are distinct, then all
entries of the block $A_{I_i \times I_j}$ lie in a single $G$-orbit.
 
\item If $i \in \{ 1, \ldots, N \}$, then all diagonal entries of
$A_{I_i   \times I_i}$ all lie in a single $G$-orbit, as do all
off-diagonal entries.
\end{enumerate}
\end{proposition}

\begin{proof}
As in the proof of Theorem~\ref{Texists}, the key step is to show that
if $\varpi_1$ and $\varpi_2$ satisfy both of the properties given
above, then $\varpi_1 \vee \varpi_2$ does too. This proceeds as in
that proof, and the only part which is not immediate the case of
off-diagonal elements of a diagonal block, with one lying above the
diagonal and the other below. However, property (2) gives that
$a_{p q} \in G a_{q p}$ if $p$ and $q$ are distinct and lie in the
same block of $\varpi_1$ or $\varpi_2$, so one may ``cross the
diagonal''.
\end{proof}

Denote the partition in Proposition \ref{Pgroupsym} by $\varpi^G( A )$
and, as above, define for each partition $\varpi$ the set
\[
\cC^G_\varpi := \{ A \in \C^{N \times N} : \varpi^G( A ) = \varpi \}.
\]

\begin{proposition}
There is a natural stratification of $\C^{N \times N}$ resulting from
Proposition~\ref{Pgroupsym}, namely
\begin{equation}\label{Eschubert2}
\C^{N \times N} = \bigsqcup_{\varpi \in \Pi_N} \cC^G_\varpi, %
\qquad \text{and} \qquad %
\overline{\cC^G_\varpi} = %
\bigsqcup_{\varpi' \prec \varpi} \cC^G_{\varpi'}.
\end{equation}
Furthermore, we have that
$\overline{\cC_\varpi^{\{ 1 \}} } = \cW_{\fG_\varpi}$.
\end{proposition}

Thus $\C^{N \times N}$ indeed admits a stratification, analogous to
the situation above. However, the reason we do not proceed further
along these lines is the lack of a \textit{rank-preserving}
equivalence of the stratum $\cC_\varpi^{\{ 1 \}}$ with any
lower-dimensional space. Indeed, if
$\varpi = \{ I_1, \ldots, I_m \}$ and
$A_\varpi := \sum_{j = 1}^m \one{I_j}$, then the
sum of $A_\varpi$ with any positive multiple of the $N \times N$
identity matrix is an element of~$\cC_\varpi^{\{ 1 \}}$ with full rank.

\subsection{Block correlation matrices and group kernels}

A popular approach for constructing probabilistic models involving
categorical data consists of grouping the input levels in such a way
that correlation is constant within each group and across groups
\cite{archakova2020,Cadima2010,HuangYang2010,roustant2017}. In
particular, the rapidly developing area of \emph{group kernels}
exploits this idea to obtain useful Gaussian processes for categorical
variables \cite{Qian2008,roustant2020}. This approach yields
parsimonious covariance models that can be naturally analyzed in the
framework of the current paper. To elaborate, consider a categorical
problem involving a potentially large number of levels $N$, {partitioned}
into a typically small number of groups $m$ of sizes $n_1$, \ldots,
$n_m$, so that $n_1 + \cdots + n_m = N$. Assuming the within-groups
correlations and between-groups correlation are constant, the
associated covariance matrix $A$ can be written in block form
\begin{equation}\label{Eblock}
A = ( A_{i j} )_{i, j = 1}^m, 
\end{equation}
where the diagonal blocks $A_{i i} \in \R^{n_i \times n_i}$ are
compound symmetry matrices of the form
\[
A_{i i} = \Id_{n_i} + c_{i i} ( \one{n_i} - \Id_{n_i} )
\]
and the off-diagonal blocks $A_{i j} \in \R^{n_i \times n_j}$ are of
the form
\[
A_{i j} = c_{i j} \bJ_{n_i \times n_j}.
\]
As shown in \cite{roustant2017,roustant2020}, positive definiteness of
$A$ is equivalent to the positivity of the much smaller compressed
matrix $\down( A )$, where $\pi$ is the partition of
$\{ 1, \ldots, N \}$ associated with the block structure of
$A$. Earlier versions of this result may be found in
\cite{Cadima2010} and \cite{HuangYang2010}.

\begin{theorem}[{\cite[Theorem~2]{roustant2017}}]\label{TblockCorr}
Let $A = ( A_{i j} )_{i, j = 1}^m$ be as above, with $0 < c_{i j} < 1$
for $i$, $j = 1$, \ldots, $m$. Then $A$ is positive definite if and
only if its compression $\down( A )$ is positive definite.
\end{theorem}

In fact, Theorem \ref{TblockCorr} is a consequence of the following
more general result. Recall the Loewner order: if $A$ and $B$ are { Hermitian} 
square complex matrix of the same size, then $A \geq B$ if and only if
$A - B$ is positive semidefinite.

\begin{theorem}[{\cite[Theorem~1]{roustant2017}}]\label{TblockCorrGen}
Let $B = ( B_{i j} )_{i, j = 1}^m$ be an arbitrary {Hermitian} block matrix with
$B_{i j} \in \R^{n_i \times n_j}$ for $i$, $j = 1$, \ldots, $m$.
Denote by $\pi$ the partition of $\{ 1, \ldots, n_1 + \cdots + n_m \}$
associated with the block structure of $B$ and let
$\overline{B_{i i}} := \down( B_{i i} ) \in \R$ be the arithmetic mean
of the entries of the block~$B_{i i}$. If
\[
B_{i i} \geq \overline{B_{i i}} \one{n_i} \qquad ( i = 1, \ldots, m )
\]
then the following are equivalent.
\begin{enumerate}
\item The matrix $B$ is positive semidefinite.
  
\item The matrix $\down( B )$ and the diagonal blocks $B_{1 1}$,
\ldots, $B_{m m}$ are positive semidefinite.
\end{enumerate}
The result also holds when ``positive semidefinite'' is replaced by
``positive definite'' in both (1) and (2).
\end{theorem}

Note that Theorem \ref{TblockCorr} follows from
Theorem~\ref{TblockCorrGen} because
\[
A_{i i} - \overline{A_{i i}} \one{n_i} = %
( 1 - c_{i i} ) ( \Id_{n_i} - n_i^{-1} \one{n_i} )
\]
is positive semidefinite when $0 < c_{i i} < 1$. More generally, the
same result holds if the diagonal blocks have the form
\[
A_{i i} = d_i \Id_{n_i} + c_{i i} ( \one{n_i}  - \Id_{n_i} ), \qquad %
\text{where } d_i \geq c_{i i}.
\]

For completeness, we provide a self-contained proof of
Theorem~\ref{TblockCorrGen} in the language developed above.

\begin{proof}[Proof of Theorem~\ref{TblockCorrGen}]
The implication (1) $\implies$ (2) is clear with the help of
Theorem~\ref{Tentrywise}. Conversely, if $C := \down( B )$ is
positive semidefinite, then so is $\up( C )$, by
Theorem~\ref{Tentrywise} again. The block-diagonal matrix
\[
D = \diag( B_{1 1} - \overline{B_{1 1}} \one{n_1}, \ldots, %
B_{m m} - \overline{B_{m m}} \one{n_m})
\] 
is positive semidefinite by assumption, and so $B = \up( C ) + D$ is
positive semidefinite. This shows that $(2) \implies (1)$.

For the positive-definite version, that $(1) \implies (2)$ is again
clear, with the help of (\ref{Edown}) and the fact that $W_\pi$ there
has full rank. For the converse, suppose $C := \down( B )$ is positive
definite and note that $B = \up( C ) + D$ is the sum of positive
semidefinite matrices. Thus, if $\bv$ is such that $\bv^T B \bv = 0$
then~$\bv$ lies in the kernels of $D$ and $\up( C )$. Since each
$B_{i i}$ is invertible, and rank is subadditive, the kernel of each
block of $D$ is at most one dimensional. Hence, the kernel of $D$ is
spanned by vectors of the form $\up( \be_i )$ for $i \in I$, where
$\be_1$, \ldots, $\be_m$ is the canonical basis of $\R^m$ and
$I \subset \{ 1, \ldots, m \}$. However, by
Proposition~\ref{Pupdownvec}(5), we have that
\[
\up( C ) \sum_{i \in I} v_i \up( \be_i ) = %
\up( C D_\pi \sum_{i \in I} v_i \be_i ) = \bZ \iff
\sum_{i \in I} v_i e_i = \bZ \qquad ( v_i \in \R )
\]
since $C$ and $D_\pi$ are invertible. Hence $D$ has trivial kernel and
$B$ is positive definite.
\end{proof}

An explicit expression for the eigendecomposition of the block
matrices featuring in Theorem \ref{TblockCorr} was obtained in
\cite{Cadima2010} and \cite{HuangYang2010}, as well as in
\cite{archakova2020} for the case where the diagonal blocks are of the
form
\[
A_{i i} = d_i \Id_{n_i} + c_{i i} ( \one{n_i} - \Id_{n_i} ).
\]
We provide a proof of the last result using our stratification
language and some results from Section~\ref{Sfuncalc}.

\begin{theorem}[{\cite[Theorem 1]{archakova2020}}]
Let $A = ( A_{i j} )_{i, j = 1}^m$ be a real matrix with
\[
A_{i j} = \begin{cases}
 d_i \Id_{n_i} + c_{i i} ( \one{n_i} - \Id_{n_i} ) & %
\text{if } i = j, \\
 c_{i j} \bJ_{n_i \times n_j} & \text{if } i \neq j.
\end{cases}
\]
Then $A = Q D Q^T$, where $Q$ is an orthogonal matrix,
\[
D = A' \oplus \bigoplus_{j = 1}^m ( d_j - c_{j j} ) \Id_{n_j - 1}
\]
and $A' = ( a'_{i j} )_{i, j = 1}^m$ with
\[
a'_{i j} = \begin{cases}
 d_i + c_{i i} ( n_i - 1 ) & \text{if } i = j, \\
 c_{i j} \sqrt{n_i n_j} & \text{if } i \neq j.
\end{cases}
\]
\end{theorem}

\begin{proof}
As above, let $\overline{A_{ii}}$ denote the arithmetic mean of the
entries of the block $A_{ii}$, so that
$\overline{A_{ii}} = n_i^{-1} ( d_i + ( n_i - 1 ) c_{ii} )$ for
$i = 1$, \ldots, $m$, and let $\pi = \{ I_1, \ldots, I_m \}$ be the
partition associated with the block structure of $A$. Then
\begin{align}\label{Etemp}
 A & = \ccup( A' ) + %
 ( A_{1 1} - \overline{A_{1 1}} \bJ_{n_1 \times n_1} ) \oplus \cdots %
\oplus ( A_{m m} - \overline{A_{m m}} \bJ_{n_m \times n_m} )
\nonumber \\
 & = \ccup( A' ) + %
\Bigl( \bigoplus_{j = 1}^m ( d_j - c_{j j} ) \Id_{n_j} \Bigr) %
\Bigl( \bigoplus_{j = 1}^m %
( \Id_{n_j} - n_j^{-1} \bJ_{n_j \times n_j} ) \Bigr).
\end{align}
We claim that 
\begin{equation}\label{Eclaim}
\sigma( A ) = %
\sigma( A' ) \cup \{ d_1 - c_{1 1}, \ldots, d_m - c_{m m} \},
\end{equation}
where $d_j - c_{j j}$ has multiplicity $n_j - 1$ for $j = 1$, \ldots,
$m$. If this holds then, since $A$ and $D$ are symmetric and have the
same spectrum, there exists an orthogonal matrix $Q$ such that
$A = Q D Q^T$.

It remains to show the claim~(\ref{Eclaim}), and we do so by
explicitly producing an orthonormal eigenbasis. First, let $V_j$ be
the orthogonal complement of~$\bJ_{I_j}$ in~$\R^{I_j}$, padded by
zeros {to} form a subspace of $\R^N$, where $N := n_1 + \cdots + n_m$.
A direct calculation shows that every non-zero vector in $V_j$ is an
eigenvector of $A$ with eigenvalue $d_j - c_{j j}$. As the
subspaces $V_1$, \ldots, $V_m$ are pairwise orthogonal, this
eigenvalue has multiplicity $n_j - 1$ as required.

Next, let $\bw_1$, \ldots, $\bw_m \in \R^m$ be an orthonormal
eigenbasis for $A'$, with eigenvalues
$\lambda_1$, \ldots, $\lambda_m$. Then
$\{ \ccup( \bw_1 ), \ldots, \ccup( \bw_m ) \}$ is an
orthonormal set, by Proposition~\ref{Pupdownvec}(3), which lies in the
span of $\{ \bJ_{I_1}, \ldots, \bJ_{I_n} \}$, by definition, so is
orthogonal to $V_i$ for $i = 1$, \ldots, $m$. Furthermore,
if $\bv \in \R^m$ then Proposition~\ref{Pupdownvec}(5) implies that
\[
\bigoplus_{j = 1}^m \Bigl( \Id_{n_j} - n_j^{-1} \one{n_j} ) \Bigr) %
\ccup( \bv ) = ( \Id_N - \ccup( \Id_m ) ) \ccup( \bv ) = \bZ_N.
\]
Hence, by~(\ref{Etemp}) and Proposition~\ref{Pupdownvec}(5), we
conclude that
\[
A \ccup( \bw_j ) = \ccup(A') \ccup( \bw_j ) + \bZ_N = %
\ccup( A' \bw_j ) = \lambda_j \ccup( \bw_j ).
\qedhere
\]
\end{proof}

In the same spirit as here, several important matrix calculations
involving the matrix~$A$, including powers, exponential, logarithm,
and Gaussian log likelihood, can be performed by working with the
compressed matrix $A'$. See \cite{archakova2020} for more details.

\subsection{Transversal matrix structures}

The classical structures of Hankel and Toeplitz matrices interact with
isogenic block stratification in a very rigid manner, as we explain below.

Consider first a positive semidefinite Hankel matrix
$H = ( c_{j + k} )_{j, k = 0}^N$ with real entries. It is well known
(see, for example, \cite{Akhiezer}) that there exists at least one
positive measure $\mu$ on the real line having $c_j$ as power moments:
\[
c_j = \int_\R x^j \std\mu( x ) \qquad ( j = 0, \ldots, 2 N ).
\]
Then, if $H$ belongs to a stratum other than the topmost,
$\stratum{\pi_\vee}$, there are distinct indices $j$ and $k$ such that
\[
c_{2 j} = c_{j + k} = c_{2 k}.
\]
Consequently,
\[
\int_\R ( x^j - x^k )^2 \std\mu( x ) = 0,
\]
so the measure $\mu$ is supported by at most three points, $-1$, $0$
and $1$. It follows for $N \geq 2$ that $H$ has rank at most $3$ and
extends uniquely to an infinite positive semidefinite Hankel matrix,
whereas if $N = 1$ then $c_0 = c_1 = c_2$ and
\[
\int_\R ( 1 - x )^2 \std\mu( x ) = 0,
\]
so the measure $\mu$ is a point mass at $1$.

A similar situation arises from the analysis of a positive
semidefinite Toeplitz matrix $T = ( s_{k - j} )_{j, k = 0}^N$. Here,
we assume the entries to be complex and the positivity of $T$ implies
that
\[
s_0 \geq 0 \qquad \text{and} \qquad %
s_{-j} = \overline{s_j} \quad ( j = 0, \ldots, N ).
\]
The solution to the truncated trigonometric moment problem guarantees
a positive measure $\nu$ on $[ -\pi, \pi )$ such that
\[
s_j = \int_{-\pi}^\pi e^{j \theta} \std\nu( \theta ) %
\qquad ( j = -N, \ldots, N ).
\]
If, as before, $T$ has some non-trivial isogenic block structure, then
there exists an index $j$ satisfying
\[
\int_{-\pi}^\pi e^{-j \theta} \std\nu( \theta ) = %
\int_{-\pi}^\pi \std\nu( \theta ) = %
\int_{-\pi}^\pi e^{j \theta} \std\nu( \theta ).
\]
It follows that
\[
\int_{-\pi}^\pi | 1 - e^{j \theta} |^2 \std\nu( \theta ) = 0,
\]
so the measure $\nu$ is the sum of at most $j$ point masses, situated
on the vertices of the regular polygon determined by that equation
$z^j = 1$. Thus, the matrix $T$ is degenerate and extends uniquely, by
periodicity to an infinite positive semidefinite Toeplitz matrix.

\subsection{Stability of semigroups}

The simple observation that the spectral radius of a square matrix
belonging to the closure of a certain stratum is preserved by the
push-down operation $\ccdown$ immediately resonates with stability
criteria of evolution semigroups. To be more specific, consider the
time-invariant linear homogeneous system of differential equations,
\[
\frac{\rd u}{\rd t}(t) = A u( t ),
\]
where $A$ is an element of $\C^{N \times N}$ and
$u : [ 0, \infty ) \to \C^N$. As is well known, the solution
\begin{equation}\label{time-independent}
u( t ) = \exp( t A ) u( 0 )
\end{equation}
decays exponentially as $t \to \infty$, for arbitrary initial data
$u( 0 )$, if and only if the spectrum of $A$ lies in the open left
half-plane. In this case, the system is called asymptotically stable.

A great deal of work has been done to establish asymptotic-stability
criteria for linear systems, in both the time-invariant and then
time-dependent cases. A relatively early work in this area is
\cite{HP}, where the distance from a stable matrix to the unstable
region is described definitively.

The isogenic block stratification introduced in the present article
can be used in some cases to simplify the verification of stability.
If the Cayley transform $X = ( A - I ) ( A + I )^{-1}$ of $A$ belongs
to a closed stratum $\overline{\stratum{\pi}}$ and the spectrum of $X$
belongs to the open unit disk, {then} so does that of the
compression $\ccdown( X )$. Furthermore, the converse is true, which
gives the following asymptotic-stability criterion.

\begin{proposition}
Suppose the Cayley transform $X$ of the generator $A$ of the linear
system (\ref{time-independent}) belongs to the stratum
$\overline{\stratum{\pi}}$. The system is asymptotically stable if and
only if the spectral radius of the compression $\ccdown( X )$ is less
than~$1$.
\end{proposition}

Much more involved, but well studied due to its important applications
to control theory, is the case of a time-varying linear system
\[
\frac{\rd u}{\rd t}( t ) = A(t) u(t).
\]
As before, it may be beneficial to examine the compression of the
Cayley transform of the family of matrices $A( t )$, but we do not
enter into the details here.

One of the challenges of modern stability theory of switched dynamical
systems is the computation of the joint numerical radius of a tuple of
non-commuting matrices. More precisely, if $A_1$, $A_2$, \ldots, $A_n$
are complex $N \times N$ matrices, one wants to certify that
\[
\lim_{m \to \infty} \max_{\sigma \in \{ 1, 2, \ldots, n\}^m} %
\| A_{\sigma( 1 )} A_{\sigma( 2 )} \cdots A_{\sigma( m )}\|^{1 / m} < 1.
\]
While the general case of the corresponding weak inequality is known
to be undecidable, sufficient conditions for the strong inequality are
known and widely used; see \cite{PJ} and references therein. Once
again, being fortunate enough to have the matrices $A_1$, $A_2$,
\ldots, $A_n$ in a high-codimension isogenic stratum may considerably
simplify, via compression, the verification of such criteria.

\subsection{Hierarchical matrices}

With its origins in the theory of numerical approximation of integral
equations, a novel chapter of applied linear algebra emerged in the
last two decades having the concept of \emph{hierarchical matrix} at
its center \cite{Bebendorf,Hackbusch-2015}. The philosophy behind this
powerful new tool is to exploit, via effective numerical schemes, the
sparsity and hidden redundancy in a large matrix of relative low
rank. Starting with the factorization of a large matrix into a product
of lower-rank rectangular matrices of smaller size, a systematic study
of minimizing computational-cost algorithms for matrix multiplication,
polar decomposition, Cholesky decomposition and much more were devised
by Hackbusch, his disciples and an increasing number of followers. The
foundational work \cite{Hackbusch-1999} is complemented by the
informative survey \cite{Hackbusch-2016}.

The matrices belonging to the closure $\overline{\stratum{\pi}}$ of a
high-codimension stratum defined by the partition
$\pi = \{ I_1, I_2, \ldots, I_m \}$ of $\{ 1, \ldots, N \}$ are
hierarchical in an obvious way. Moreover, the compression
$\ccdown : \overline{\stratum{\pi}} \to \C^{m \times m}$ naturally
aligns with the main concepts of the hierarchical matrix theory. We
indicate here a few common trends.

The isogenic structure of $A \in \overline{\stratum{\pi}}$ is
reflected by the factorization
\[
A = \cW_\pi B \cW_\pi^*,
\]
where $B$ is an $m \times m$ matrix and $\cW_\pi \in \C^{N \times m}$;
see Proposition \ref{Pdownup}. Thus every element of
$\overline{\stratum{\pi}}$ can be decomposed with a universal left
factor $\cW_\pi$ and a right factor depending on $m^2$ complex
parameters. Moreover, the weighted compression $\ccdown$ built on the
same data (the partition $\pi$ and the parameters determining each
isogenic block) is a $*$-algebra homomorphism. Specifically, a matrix
$A \in \overline{\stratum{\pi}}$ has the
multiplicative decomposition
\[
A = %
\cW_\pi D_\pi^{-1} \cW_\pi^* \ \ccdown( A ) \ %
\cW_\pi D_\pi^{-1} \cW_\pi^*,
\]
and the correspondence $A \mapsto \ccdown( A )$ preserves all matrix
operations.

\begin{proposition}
If the matrix $A \in \overline{\stratum{\pi}}$, then compression
$\ccdown( A )$ has the same spectrum and singular {values} as
$A$, with the possible exception of the value zero.
\end{proposition}
\begin{proof}
Proposition~\ref{Peigen} gives that
$\sigma( A ) \setminus \{ 0 \} = %
\sigma( \ccdown( A ) ) \setminus \{0 \}$.
Moreover, the singular {values} of $A$ are the eigenvalues of
its modulus $| A | = \sqrt{A^* A}$, but $A^* A$ and
$\ccdown ( A )^* \ccdown( A )$ have the same non-zero
eigenvalues.
\end{proof} 

Lifting the polar decomposition is slightly more subtle. Let
$B = \ccdown( A )$, where $A \in \overline{\stratum{\pi}}$. The polar
decomposition $B = U_B | B |$ lifts immediately to~$A = V_A | A |$,
but $V_A := \ccup( U_B )$ is only a partial isometry. There is,
however, a unitary matrix $U_A \in \C^{N \times N}$ such that
that $U_A | A | = V_A | A |$. The initial and final spaces of $V_A$
are equal and contain the range of $A$, and are spanned by
$\{ \bJ_{I_k} : k = 1, \ldots, m \}$. If $W$ is the orthogonal
projection onto the orthonogal complement of this space then
$U_A = V_A + W$ is as desired.

The lifting of an LU decomposition is similarly not quite immediate.
As before, let $B$ be the compression of the matrix
$A \in \overline{\stratum{\pi}}$ and suppose $B = L U$, where $L$ is a
lower-triangular matrix and $U$ is upper triangular. Then the liftings
$\ccup( L )$ and $\ccup( U )$ will be only  block triangular, being
elements of $\overline{\stratum{\pi}}$. To obtain a genuine LU
decomposition, one has to depart from the isogenic structure and split
the products of rank-one matrices in a consistent pattern.  For
example, one can factor the product of $k \times n$ and $n \times m$
isogenic matrices as follows:
\[
\begin{pmatrix}
a & \cdots & a\\
\vdots & \ddots & \vdots\\
a & \cdots & a
\end{pmatrix} \begin{pmatrix}
b & \cdots & b\\
\vdots & \ddots & \vdots\\
b & \cdots & b
\end{pmatrix} = \begin{pmatrix}
a \sqrt{n} & 0 & \cdots & 0\\
\vdots & & \ddots & \vdots\\
a \sqrt{n} & 0 & \cdots & 0
\end{pmatrix} \begin{pmatrix}
b \sqrt{n} & \cdots & b \sqrt{n}\\
0 & \cdots & 0\\
\vdots & \ddots & \vdots\\
0 & \cdots & 0
\end{pmatrix}.
\]

Finally, on the lighter side, we state the following theorem to give a
demonstration of the general paradigm for results one can derive using
the compression-inflation procedure described above.

\begin{theorem} 
Let $A$, $B$ and $C$ be real matrices belonging to
$\overline{\stratum{\pi}}$, where $\pi$ is a partition of
$\{ 1, 2, \ldots, N \}$. The classical equations of Lyapunov,
Sylvester and Riccati,
\[
X + X A^T = C \qquad %
A X + X B = C \quad \text{and} \quad %
A X + X A^T - X B X = C,
\]
have a solution $X \in \C^{N \times N}$, with $X = X^T$ for the
Riccati case, if and only if the compressed equations have a solution
in $\C^{m \times m}$.
\end{theorem}

\subsection{Coherent matrix organization}

Complex multivariate data, such as those arising from the measurement
of neuronal structures, cannot be handled without regularization and
compression of the corresponding large matrices. Among the many
techniques to analyze and transform such matrices, the method of
\emph{coherent matrix organization} proposed by Gavish and Coifman
\cite{GC} stands out for its elegance and universality. These authors
propose organizing a matrix using a natural metric which clusters the
entries by proximity; the emerging tree of finer and finer partitions
offers a canonical compression by averaging along strata, with tight
control of the resulting approximation error.

The isogenic stratification of a matrix and the compression maps
discussed above align perfectly with Gavish and Coifman's idea. Our
purely theoretical setting is an extremal one, in the sense that we
employ the discrete metric for clustering. However, it is notable that
the isogenic structure and the associated inflation map offer a rapid
matrix-completion algorithm within a prescribed stratum, in the spirit
of recent advances in coherent matrix organization; see, for instance,
\cite{CBSW, Gross, MTCCK}. We do not expand here some clear
consequences of this non-accidental similarity.



\end{document}